\renewcommand{\phi}{\varphi}
\renewcommand{\theta}{\vartheta}
\documentclass[12pt,oneside,reqno]{amsart}
 
\usepackage[all]{xy}                        %
\usepackage[hyperindex=true, colorlinks=true, linkcolor=blue, filecolor=black,
citecolor=blue, urlcolor=blue, bookmarks=true, pagebackref=true,
pdftitle={COUNTING ACM TORIC BUNDLES OF RANK 2
ON SMOOTH TORIC SURFACES}]{hyperref}

\usepackage[margin=1in]{geometry}
\usepackage[table]{xcolor}
\usepackage{colortbl, array, booktabs}

\definecolor{lightred}{rgb}{1, 0.85, 0.85}
\definecolor{lightblue}{rgb}{0.85, 0.9, 1}
\definecolor{darkline}{rgb}{0.4, 0.4, 0.4} 
\usepackage{verbatim}
\usepackage[numbers]{natbib}
\CompileMatrices                            

\UseTips                                    

\usepackage[bookmarks=true]{hyperref}       
\hypersetup{
    colorlinks=true,    
    linkcolor=blue,     
    citecolor=blue,     
    filecolor=blue,     
    urlcolor=blue       
}
\usepackage{amssymb,latexsym,amsmath,amscd}
\numberwithin{equation}{section}
\usepackage{xspace}
\usepackage{color}
\usepackage{graphicx}
\usepackage{caption}
\usepackage{mathtools}
\usepackage{tikz}
\usetikzlibrary{cd}
\usepackage[utf8]{inputenc}
\usepackage{fourier}
\usepackage{array}
\usepackage{makecell}
\usepackage{colortbl}
\usepackage{bm}

\newcolumntype{L}[1]{>{\raggedright\let\newline\\\arraybackslash\hspace{0pt}}m{#1}}
\newcolumntype{C}[1]{>{\centering\let\newline\\\arraybackslash\hspace{0pt}}m{#1}}
\newcolumntype{R}[1]{>{\raggedleft\let\newline\\\arraybackslash\hspace{0pt}}m{#1}}
\definecolor{lightred}{rgb}{1.0,0.9,0.9}
\definecolor{lightblue}{rgb}{0.9,0.9,1.0}

\reversemarginpar

\theoremstyle{plain}
\newtheorem{theorem}{Theorem}[section]
\newtheorem*{theorem*}{Theorem}
\newtheorem{proposition}[theorem]{Proposition}
\newtheorem{corollary}[theorem]{Corollary}
\newtheorem{lemma}[theorem]{Lemma}

\theoremstyle{definition}
\newtheorem{definition}[theorem]{Definition}

\newtheorem{remark}[theorem]{Remark}
\newtheorem{example}[theorem]{Example}

\newcommand{\enm}[1]{\ensuremath{#1}}          %

\newcommand{\cal}[1]{\mathcal{#1}}

\newcommand{\ZZ}{\enm{\mathbb{Z}}}

\newcommand{\PP}{\enm{\mathbb{P}}}

\newcommand{\Ee}{\enm{\cal{E}}}

\newcommand{\Gg}{\enm{\cal{G}}}

\newcommand{\Ii}{\enm{\cal{I}}}

\newcommand{\Ll}{\enm{\cal{L}}}

\newcommand{\Oo}{\enm{\cal{O}}}
\renewcommand{\epsilon}{\varepsilon}

\newcommand{\tensor}{\otimes}         

\renewcommand{\to}[1][]{\xrightarrow{\ #1\ }}

\newcommand{\old}[1]{}

\begin{document}
\renewcommand{\arraystretch}{1.4}
\arrayrulecolor{darkline}

\title[Counting aCM toric bundles of rank two on the Veronese surface]{Counting aCM toric bundles of rank two\\ on the Veronese surface}

\author{Yeonjae Hong and Sukmoon Huh}

\address{Sungkyunkwan University, Suwon 440-746, Korea}
\email{mathbunny529@skku.edu}

\address{Sungkyunkwan University, Suwon 440-746, Korea}
\email{sukmoonh@skku.edu}

\thanks{SH is supported by the National Research Foundation of Korea(NRF) grant funded by the Korea government(MSIT) (No. RS-2023-00208874).}

\keywords{Toric equivariant bundle, arithmetically Cohen-Macaulay bundle}

\maketitle

\begin{abstract}
In this paper, we count the number of aCM vector bundles with a toric structure on the Veronese surface, up to a twist by hyperplane divisor class. The main ingredients are the equivalence introduced by A.~A. Klyachko between the toric reflexive sheaves and certain decreasing filtrations, and a description of the free resolution of toric vector bundles by M.~Perling. 
\end{abstract}

\section{Introduction}
Let \( X \) be a smooth projective variety of dimension \( n \) with a polarization \( \mathcal{O}_X(1) \). A vector bundle \( \mathcal{E} \) on \( X \) is called \emph{arithmetically Cohen--Macaulay} (for short, aCM) if it satisfies
\[
\mathrm{H}^i(X,\mathcal{E}(t)) = 0 \text{ for all } t \in \mathbb{Z} \text{ and } 1 \le i \le n - 1.
\]
Projective varieties can be classified according to the complexity of their indecomposable aCM bundles. If there are only finitely many such bundles (up to twist and isomorphism), then \( X \) is said to be of \emph{finite representation type}. If indecomposable aCM bundles exist in at most one-parameter families, then \( X \) is said to be of \emph{tame representation type}. Otherwise, if \( X \) admits families of arbitrary large dimension, it is called \emph{wild}. By Horrocks’ splitting criterion, the only aCM vector bundles on \( \mathbb{P}^n \) are direct sums of line bundles. Hence, projective space is of finite representation type. In \cite{CMR16}, L.~Costa and R.~M.~Miró-Roig showed that most Grassmannians are of wild representation type by constructing large families of indecomposable homogeneous aCM bundles. More generally, a trichotomy theorem due to D.~Faenzi and J.~Pons-Llopis states that every reduced, non-degenerate, arithmetically Cohen–Macaulay closed subscheme $X \subset \mathbb{P}^n$ of positive dimension admits exactly one of the three representation types mentioned above in \cite{DJ21}. Via the \( d \)-fold Veronese embedding \( \mathbb{P}^2 \hookrightarrow \mathbb{P}^{\binom{d+2}{2} - 1} \), the study of aCM bundles on the Veronese surface reduces to that of vector bundles \( \mathcal{E} \) on \( \mathbb{P}^2 \) satisfying:
\begin{equation} \label{dacm}
\mathrm{H}^1(\mathbb{P}^2,\mathcal{E}(dt)) = 0 \text{ for all } t \in \mathbb{Z}.
\end{equation}
From now on, we refer to any bundle \( \mathcal{E} \) on $\mathbb{P}^2$ satisfying $\eqref{dacm}$ as a  \textit{d -aCM} bundle. In~\cite{Fae15}, D.~Faenzi used Beilinson's theorem to classify rigid $3$-aCM bundles and showed that the 3-Veronese surface has finite representation type with respect to rigid aCM bundles.

Although \( d \)-aCM bundles are defined in a fairly general setting, we restrict our attention to those bundles that carry a torus-equivariant structure. Quasi-coherent torus-equivariant sheaves on toric varieties can be described using multifiltrations of vector spaces. This approach was initiated by A.~A. Klyachko in \cite{Kly90,Kly91} and has since been developed by various authors, including \cite{Per03,Kool11}. This framework is useful not only for computing various invariants of equivariant sheaves, such as cohomology, slope stability, and Chern classes, but also for constructing explicit examples. In particular, for toric bundles on projective space, A.~A. Klyachko has already provided a simplified formula for computing cohomology. Using this, we derive the following main theorem.
\begin{theorem}\label{main}
The number $\mathrm{S}(\PP^2,d;2)$ of non-split $d$-aCM vector bundles of rank $2$ with equivariant structure on $\PP^2$, up to twist by $\Oo_{\PP^2}(dt)$ with $t\in \ZZ$, is given by
\[
\mathrm{S}(\PP^2,d;2)=\frac{(d-1)d(d+1)(d+2)}{24}.
\]
\end{theorem}

This paper begins with a review of basic notions related to toric varieties and torus-equivariant sheaves. We introduce the necessary background and tools for understanding equivariant vector bundles in the context of toric geometry. In the later section, we focus on counting rank $2$ toric bundles on \( \mathbb{P}^2 \) of rank $2$ that satisfy the \( d \)-aCM condition. We also provide explicit descriptions of such bundles when \( d \) is small, illustrating their structure through concrete examples.


\section{Toric Geometry and Equivariant Sheaves}
\subsection{Basics on the toric variety}
In this subsection, we review some basic concepts of toric varieties; see \cite{CLS11, Ful93} for further details.
An $n$-dimensional  \textit{toric variety} is an irreducible variety $X$ containing a torus $T \simeq\left(\mathbb{C}^{\times}\right)^{n}$ as a Zariski open subset such that the action of $T$ on itself extends to an algebraic action of $T$ on $X$. Thus there exists a morphism $T \times X \rightarrow X$, that satisfies the properties required for an algebraic group action.

A \textit{character} of $T$ is a morphism $\chi: T \rightarrow \mathbb{C}^{\times}$ that is a group homomorphism. Given $m = (a_{1}, \dots , a_{n}) \in \mathbb{Z}^n$, the character $\chi^{m}: (\mathbb{C}^{\times})^{n}\rightarrow \mathbb{C}^{\times}$ defined by 
\[
\chi^{m}(t_{1},\dots,t_{n}) = {t_{1}}^{a_1} \dots {t_{n}}^{a_{n}}.
\]
The set of characters forms a free abelian group $M$ of rank $n$. A \textit{one-parameter subgroup} of $T$ is a morphism $\ u: \mathbb{C}^{\times} \rightarrow T$ that is a group homomorphism. The one-parameter subgroups form a free abelian group $N$ of rank $n$.
We define $N_{\mathbb{R}} := N \otimes_{\mathbb{Z}} \mathbb{R}$ and $M_{\mathbb{R}} := M \otimes_{\mathbb{Z}} \mathbb{R}$. There is a natural pairing \( \langle \cdot, \cdot \rangle \colon M \times N \to \mathbb{Z} \), defined by the relation
\[
m \circ u(t) = t^{\langle m, u \rangle},
\]
for all \( t \in \mathbb{C}^\times \), where \( m \in M \) and \( u \in N \).
A \textit{fan} $\Sigma$ in $N_{\mathbb{R}}$ is a collection of strongly convex rational polyhedral cones satisfying:
\begin{itemize}
    \item If $\sigma \in \Sigma$, then every face of $\sigma$ also belongs to $\Sigma$.
    \item If $\sigma, \tau \in \Sigma$, then $\sigma \cap \tau$ is a face of both.
\end{itemize}
For a given cone $\sigma \in \Sigma$, we use the notation $\tau \preceq \sigma$ to denote that $\tau$ is a face of $\sigma$. We denote by \( \Sigma(k) \) the set of all \( k \)-dimensional cones in \( \Sigma \). Elements of $\Sigma(1)$ are referred to as \textit{rays}. For a ray $\rho\in \Sigma(1)$, the semigroup $\rho \cap N$ is generated by a unique element $u_\rho \in \rho \cap N$. We call $u_\rho$ the \textit{primitive vector} of $\rho$. For any cone \( \sigma \in \Sigma \), its dual cone in \( M_{\mathbb{R}} \) is given by
\[
\sigma^\vee = \{ m \in M_{\mathbb{R}} \mid \langle m, u \rangle \geq 0 \text{ for all } u \in \sigma \},
\]
and one can construct the affine toric variety $U_{\sigma}= \operatorname{Spec}(\mathbb{C}[S_{\sigma}])$ for the semigroup $S_{\sigma}=\sigma^\vee \cap M$. There is a one-to-one correspondence between isomorphism classes of separated normal toric varieties \( X \) and isomorphism classes of fans in \( N_{\mathbb{R}} \). Through this correspondence, the cones in the fan correspond to the \( T \)-invariant affine charts $U_{\sigma}\subset X$; see \cite[Chapter 3]{CLS11}.

There exists a one-to-one correspondence between the set of rays $\rho \in \Sigma(1)$ and the set of $T_N$-invariant Weil divisors $D_\rho$. Moreover, these divisors form a generating set for the divisor class group $\operatorname{Cl}(X)$. This relationship is described by the following exact sequence:
\[
0 \to M \xrightarrow{\phi} \bigoplus_{\rho \in \Sigma(1)} \mathbb{Z} D_\rho \xrightarrow{\pi} \operatorname{Cl}(X) \to 0,
\]
where for each character $m \in M$, the map $\phi$ is given by 
\[
\phi(m) = \operatorname{div}(\chi^m) = \sum_{\rho \in \Sigma(1)} \langle m, n(\rho) \rangle D_\rho,
\]
and the projection $\pi$ maps a divisor $D$ to its class $[D]$ in $\operatorname{Cl}(X)$. Consequently, the divisor class group $\operatorname{Cl}(X)$ is a finitely generated abelian group.


\subsection{Equivariant torsion free sheaves}
Let \( X \) be a toric variety. A sheaf \( \mathcal{E} \) on \( X \) is said to be \textit{torus-equivariant} (for short, $T$-equivariant) if, for every \( t \in T \), there exists an isomorphism  
\[
\Phi_t : t^* \mathcal{E} \to \mathcal{E}
\]  
such that the following commutative diagram holds for all \( t, t' \in T \):

\[
\begin{array}{ccc}
t^* t'^* \mathcal{E} & \xrightarrow{\quad t^* \Phi_{t'} \quad} & t^* \mathcal{E} \\
 {\cong} \uparrow \quad &  & \quad \downarrow \Phi_t \\
(t't)^* \mathcal{E} & \xrightarrow{\quad \Phi_{t't} \quad} & \mathcal{E}
\end{array}
\]

\noindent Given a cone \( \sigma \), one can define a preorder $\preceq_{\sigma}$ on \( M \) via
\[
m \preceq_{\sigma} m' \quad \text{if and only if} \quad m' - m \in S_{\sigma}.
\]
This allows us to define filtrations and associated vector spaces.

\begin{definition}
A \textit{\(\sigma\)-family} $\hat{E}^{\sigma}$ consists of a collection of vector spaces \( \{E^\sigma_{ m} \}_{m \in M} \) equipped with transition maps
\[
\chi^\sigma_{m, m'}: E^\sigma_{m} \to E^\sigma_{m'}
\]
for each \( m \preceq_{\sigma} m' \), satisfying the conditions:
\begin{enumerate}
    \item \( \chi^\sigma_{m, m} \) is the identity map.
    \item If \( m \preceq_{\sigma} m' \preceq_{\sigma} m'' \), then \( \chi^\sigma_{m, m''} = \chi^\sigma_{m', m''} \circ \chi^\sigma_{m, m'} \).
\end{enumerate}
\end{definition}

Every $T$-equivariant quasi-coherent sheaf over \( U_{\sigma} \) gives rise to a corresponding \(\sigma\)-family through its isotypical decomposition. Specifically, one sets
\[
E^\sigma_{ m} = \Gamma(U_{\sigma}, \mathcal{E})_m
\]
where the right-hand side denotes the \( m \)-th isotypical component. The transition maps are given by multiplication with monomials in the coordinate ring $\mathbb{C}[S_{\sigma}]$. According to 
\cite[Theorem 4.5]{Per03}, the category of $T$-equivariant quasi-coherent sheaves over $U_{\sigma}$ is equivalent to the category of $\sigma$-families. To describe an $T$-equivariant quasi-coherent sheaf in terms of filtrations on an abstract toric variety, we consider the following definition, which corresponds to the gluing of $\sigma$-families. 
\begin{definition}\cite[Definition 4.8]{Per03} Let \( \Sigma \) be a fan. A collection \( \{ \hat{E}^{\sigma} \}_{\sigma \in \Sigma} \) of \(\sigma\)-families is called a \(\Sigma\)-\textit{family}, if for each pair \( \tau \prec \sigma \) with inclusions 
$
i_{\sigma}^{\tau} : U_{\tau} \hookrightarrow U_{\sigma}
$
there exists an isomorphism of families 
$
\eta_{\tau\sigma} : i_{\sigma}^{\tau*} \hat{E}^{\sigma} \cong \hat{E}^{\tau}
$
such that for each triple \( \rho \prec \tau \prec \sigma \), the following equality holds:
$
\eta_{\rho\sigma} = \eta_{\rho\tau} \circ {i_{\tau }^{\rho}}^* \eta_{\tau\sigma}.
$
\end{definition}

If the sheaf \(\mathcal{E}\) is torsion-free, all vector spaces within the \(\Sigma\)-family can be regarded as subspaces of the stalk of \(\mathcal{E}\) at the generic point. Throughout the paper, we denote the stalk of a sheaf \(\mathcal{E}\) at the generic point by \(\mathrm{E}\), using roman font to distinguish it from the sheaf itself. This convention also applies to other sheaves such as \(\mathcal{F}\), \(\mathcal{G}\), etc.  
Further, if the sheaf is reflexive, the \(\sigma\)-families for \(\sigma \in \Sigma(1)\) determine the rest; see \cite[Theorem 4.21]{Per03}.

We assume that $\dim (X)=\dim (\Sigma)=n$.  We set rays by $\{\rho_0, \dots, \rho_m\}$. And let $u_i$ be the primitive vector for the ray $\rho_i$.
For a fixed vector space $\mathrm{E}$, we consider a collection of decreasing $\ZZ$-filtrations $\{E^{\rho_i}(\bullet) \}_{{\rho_i} \in \Sigma(1)}$ with 
\[
E^{\rho_i}(\bullet): \quad \mathrm{E}\supseteq \dots \supseteq E^{\rho_i}(-1) \supseteq E^{\rho_i}(0) \supseteq E^{\rho_i} (1) \dots \supseteq 0
\]
for each ray $\rho_i\in \Sigma(1)$. The filtrations $\{E^{\rho_i} (\bullet)\}_{{\rho_i} \in \Sigma(1)}$ are said to be {\it full}, if $E^{\rho_i}(j)=0$ for all $j$ with $|j|\gg0$ and each $i$. Without any specification, the filtrations are always assumed to be full. Then A.~A. Klyachko presented a theorem to express $T$-equivariant reflexive sheaves as a set of decreasing $\ZZ$-filtrations.

\begin{theorem}\textnormal{\cite[Theorem 1.3.2]{Kly91}}\label{kly} The category of $T$-equivariant reflexive sheaves on $X$ is equivalent to the category of decreasing $\mathbb{Z}$-filtrations $\{E^{\rho_i} (\bullet)\}_{{\rho_i} \in \Sigma(1)}$. Such a reflexive sheaf is a toric bundle if and only if the corresponding filtrations satisfy the following compatibility condition:

\leftskip=2em 
\noindent $\mathrm{(C)}$ For any cone $\sigma \in \Sigma$, the filtrations $\{E^{\rho_i} (\bullet)\}_{{\rho_i} \in \sigma(1)}$, consist of coordinate subspaces of some basis of $\mathrm{E}$. 
\rightskip=3em
\end{theorem}

\noindent In Theorem \ref{kly}, we have $\dim (\mathrm{E})=\mathrm{rank}(\Ee)$.
Let $r=\mathrm{rank}(\Ee)$ and define the following number
\[
\eta(\Ee; i):=\left \{\dim E^{\rho_i}(j)\mid j\in \ZZ\right \}\subset \{0,1,\dots, r\}. 
\]
Let $s_i := |\eta(\Ee; i)| - 1$, and set $s := s_i$. Then $\eta(\Ee; i)=\{d_0, \dots, d_s\}$ with $0=d_0<d_1<\dots <d_s=r$. Since the filtrations of \( \mathcal{E} \) are assumed to be full, we have \( \{0, r\} \subset \eta(\mathcal{E}; i) \) for all \( i \).  
Let \( W_{d_j}^i \) denote the subspace \( E^{\rho_i}(t) \) for some \( t \in \mathbb{Z} \) such that \( \dim E^{\rho_i}(t) = d_j \).  
Then there exist integers \( a_i^t \in \mathbb{Z} \) for \( 1 \le t \le s \) such that the filtration is given by:

\begin{equation}\label{eq:standard}
E^{\mathbf{\rho}_{i}}(j)=\begin{cases}
\hspace{.2cm} W_i^{d_s} = \mathrm{E} & \text{for}\quad j \leq a_{i}^{d_s} \vspace{.3cm}\\
\hspace{.2cm} W_i^{d_{s-1}} & \text{for}\quad a_{i}^{d_s}< j \leq a_{i}^{d_{s-1}} \vspace{.1cm}\\
\hspace{.6cm} \vdots & \hspace{1.8cm} \vdots \vspace{.1cm}\\
\hspace{.2cm} W_i^{d_1} & \text{for}\quad a_{i}^{d_2}< j \leq a_{i}^{d_1} \vspace{.3cm}\\
\hspace{.2cm} W_i^{d_0}=0 & \text{for}\quad a_{i}^{d_1}< j
\end{cases}
\end{equation}
Then, the filtration $E^{\rho_i}(\bullet)$ is essentially given by a flag 
\[
\mathrm{E}=W_i^{d_s} \supsetneq W_i^{d_{s-1}}\supsetneq \dots \supsetneq W_i^{d_1} \supsetneq W_i^{d_0}=0
\]
with $\dim W_{i}^{d_t}=d_t$. Note that, if $\eta(\Ee; i)= \{0,1,\dots, r\}$, the flag is full and so we get 
$\text{dim}(W_{i}^{d_t}) = t$. 

\begin{remark}
Let $\mathcal{E}$ and $\mathcal{F}$ be toric bundles with the associated filtrations $\{E^{\rho_i} (j)\}_{\scriptscriptstyle{\rho_i} \in \Sigma(1)}$ 
of $\mathrm{E}$ and $\{F^{\rho_i} (j)\}_{\scriptscriptstyle{\rho_i} \in \Sigma(1)}$ of $\mathrm{F}$ respectively. Then, the filtrations associated with $\mathcal{E} \oplus \mathcal{F}$ and  $\mathcal{E} \otimes \mathcal{F}$ are given as follows:
\[
\begin{aligned}
& \left\{(E \oplus F)^{\rho_{i}}(\bullet)\right\}_{\scriptscriptstyle{\rho_i} \in \Sigma(1)}, \text { where }(E \oplus F)^{\rho_i}(j)=E^{\rho_i}(j) \oplus F^{\rho_i}(j), \\
& \left\{(E \otimes F)^{\rho_{i}}(\bullet)\right\}_{\scriptscriptstyle{\rho_i} \in \Sigma(1)}, \text { where }(E \otimes F)^{\rho_i}(j)=\sum_{s+t=j} E^{\rho_i}(s) \otimes F^{\rho_i}(t);
\end{aligned}
\]
refer to \cite[Remark 2.2.15]{DDK20}. Note that the filtrations in the above are given with subspaces of $\mathrm{E} \oplus \mathrm{F}$ and $\mathrm{E} \otimes \mathrm{F}$, respectively.
\end{remark}

\begin{example} \label{Twist}
Given a $T$-invariant divisor $D=\ell_{0} D_{0}+\ell_{1} D_{1}+\cdots+\ell_{m} D_{m} \in$ $\operatorname{Div}_{T}(X)$, the associated line bundle $\Ll = \mathcal{O}_{X}(D)$ is naturally a toric bundle with the corresponding filtrations $\{L^{\rho_i}(\bullet)\}_{{\rho_i} \in \Sigma(1)}$
\[
L^{\mathbf{\rho}_{i}}(j)= 
\begin{cases}
\hspace{0.2cm} \mathrm{L} \cong \mathbb{C} & \mbox{for}\quad j \leq \ell_i  \vspace{0.3cm}\\ 
\hspace{0.2cm} 0 & \mbox{for}\quad j > \ell_i
\end{cases}
\]
\noindent Let $\mathcal{E}$ be a toric bundle whose associated filtrations are as in (\ref{eq:standard}). Then the filtrations $\{G^{\rho_i}(\bullet)\}_{\scriptscriptstyle{\rho_i} \in \Sigma(1)}$ of $\Gg: = \mathcal{E} \otimes\Ll$ are given by
\[
G^{\mathbf{\rho}_{i}}(j)=\begin{cases}
\hspace{.2cm} W_i^{d_s} = \mathrm{G} & \text{for}\quad j \leq a_{i}^{d_s} + \ell_i \vspace{.3cm}\\
\hspace{.2cm} W_i^{d_{s-1}} & \text{for}\quad a_{i}^{d_s} + \ell_i< j \leq a_{i}^{d_{s-1}} + \ell_i \vspace{.1cm}\\
\hspace{.2cm} \vdots & \vdots \vspace{.1cm}\\
\hspace{.2cm} W_i^{d_1} & \text{for}\quad a_{i}^{d_2} + \ell_i< j \leq a_{i}^{d_1} + \ell_i \vspace{.3cm}\\
\hspace{.2cm} W_i^{d_0}=0 & \text{for}\quad a_{i}^{d_1} + \ell_i< j
\end{cases}
\]

\noindent In other words, tensoring with a line bundle increases the indices of the filtration corresponding to the ray $\rho_i$ by the coefficients of $D_i$. Therefore, the sum of the coefficients $\sum_{\substack{i}} \ell_i$ represents the total change in indices across all filtrations.

Throughout this paper, we say that a bundle splits if it is isomorphic to a direct sum of line bundles.  
By comparing the filtrations associated to a direct sum with the direct sum of filtrations, we can determine when a toric bundle splits.
\end{example}

\begin{corollary} \label{toric_split}
\textnormal{\cite[Corollary 2.3.3]{Kly90}} A toric bundle $\Ee$ on $X$ splits if and only if its corresponding filtrations 
\[\{E^{\rho_i} (j)\}_{\scriptscriptstyle{\rho_i} \in \Sigma(1),j \in \mathbb{Z}}
\]
consist of coordinate subspaces of a basis for the space.
\end{corollary}

\begin{example}
The fan of \( \mathbb{P}^1 \) has two rays, \( \rho_0 \) and \( \rho_1 \).  
For any pair of filtrations of a vector space \( E \), one can choose a subbasis of \( E \) that simultaneously represents all the subspaces in both filtrations.  
As a consequence, every toric bundle on \( \mathbb{P}^1 \) necessarily splits. A.~A.~Klyachko showed that, more generally, a toric bundle of rank \( r \) on \( \mathbb{P}^n \) splits only if \( r < n \); see \cite[Example~6.3.3]{Kly90}.
\end{example}

\begin{example}
Let $\Ee$ be a toric bundle over $\mathbb{P}^2$ with rank 2. The fan of $\mathbb{P}^2$ has three rays. If at most two distinct $1$-dimensional vector spaces appear in the filtrations of $\Ee$, we can select a basis of E that includes their generators. So if $\Ee$ does not split, then there must exist three distinct $1$-dimensional vector spaces within the filtrations.
\end{example}

For a smooth toric variety \( X \), the tangent bundle \( \mathcal{T}_X \) is also \( T \)-equivariant, and thus one can consider the corresponding filtrations.

\begin{example} \label{Tangentsheaf} \cite[Example 2.1.3]{KD19} The filtrations associated to the tangent bundle \( \mathcal{T}_X \) are given by:
\[
E^{\mathbf{\rho}_{i}}(j)=
\begin{cases}
\hspace{.2cm} \mathrm{E} & \text{for } \quad j \leq 0 \\
\hspace{.2cm} \text{Span}(u_{\rho_i}) & \text{for } \quad j = 1 \\
\hspace{.2cm} 0 & \text{for } \quad j > 1.
\end{cases}
\]
\end{example}

\begin{remark} \label{rmk_210}
There is an infinite number of toric bundles on $X$ sharing the same underlying vector bundle structure. Indeed, tensoring with trivial line bundle associated to a nonzero $T$-invariant principal divisor yields an isomorphic vector bundle but a non-isomorphic toric bundle.
Conversely, according to \cite[Corollary 1.2.4]{Kly90},  if two toric bundles $\Ee$ and ${\Ee}'$ are isomorphic as vector bundles, then there exists a $T$-invariant principal divisor $D$ such that $\Ee \cong {\Ee}' \otimes \mathcal{O}_{X}(D)$ as toric bundles.
\end{remark}

In \cite[pages 8 and 9]{Kly91}, the Chern class and cohomology group of a toric bundle $\Ee$ are computed in terms of the corresponding filtrations; for any $\sigma \in \Sigma$ and $m \in M$, set
\begin{equation} \label{eq:klyachko_graded}
E^{\sigma}_{m} := \bigcap_{\rho_i \in \sigma(1)} E^{\rho_i}(\langle m, u_i \rangle), \qquad
E^{[\sigma]}_{m} := \frac{E^{\sigma}_{m}}{\sum\limits_{(k_\rho)} \left( \bigcap\limits_{\rho \in \sigma(1)} E^{\rho}(k_\rho) \right)},
\end{equation}

\noindent where the sum runs over tuples $(k_\rho)_{\rho \in \sigma(1)}$ satisfying 
$k_\rho \geq \langle m, \rho \rangle$ for all $\rho$, with strict inequality for at least one. 
Although the collection is infinite, only finitely many terms contribute due to the decreasing nature of the filtrations. One can see from \cite[Theorem 3.2.1]{Kly90} that the total Chern class of a toric bundle $\mathcal{E}$ is determined by the image in the Chow ring of the following polynomial.

\begin{equation}
\mathit{c}(\mathcal{E})=\prod_{\sigma \in \Sigma} \prod_{m \in M_{\sigma}}\left(1+\sum\limits_{\mathbf{\rho}_{i} \in \sigma(1)}\left\langle m, \mathbf{\rho}_{i}\right\rangle x_{i}\right)^{(-1)^{\operatorname{codim} \sigma} \operatorname{dim} E^{[\sigma]}_{m}}
\label{eq:totalchern}
\end{equation}

\noindent By considering the first graded component of this equation, we obtain an expression for the first Chern class of $\Ee$ in terms of its filtrations:
\begin{equation}
\mathit{c}_{1}(\mathcal{E})=\sum_{i=1}^{m+1}\left(\sum_{j \in \mathbb{Z}} j \operatorname{dim} {E}^{\left[\mathbf{\rho}_{i}\right]}(j)\right) D_{i}.
\label{eq:firstchern}
\end{equation}

\noindent To study the cohomology group, we consider the following chain complex $C_{*}(\mathcal{E}, m)$:
\[
0 \leftarrow E \leftarrow \bigoplus_{\substack{\operatorname{dim} \sigma=1}} E^{\sigma}_{m} \leftarrow \bigoplus_{\substack{\operatorname{dim} \sigma=2}} E^{\sigma}_{m} \leftarrow \cdots \leftarrow \bigoplus_{\substack{\operatorname{dim} \sigma=n}} E^{\sigma}_{m} \leftarrow 0.
\]
Now, we fix an orientation of the simplicial complex arising from the intersection of sphere $S^{n-1}$ and fan. So, for every cone in the fan, there is a canonical order among its faces. The differential map of the above chain complex is given by the formula
\[
d^{k}=\sum_{\operatorname{dim} \sigma=k} d^{\sigma} ; \quad d^{\sigma}=\sum_{i}(-1)^{i} \varepsilon_{i}
\]
where $\sigma_{i} =  (\alpha_1,\, \ldots \,,\, \hat{\alpha}_i,\, \ldots \,,\, \alpha_k)$ is $i$-th face  of $\sigma = (\alpha_1,\, \ldots \,,\, \alpha_k)$ and $\varepsilon_{{i}}: E^{\sigma_{i}}_{m} \hookrightarrow E^{\sigma}_{m}$ is the natural inclusion.

\begin{theorem}\textnormal{\cite[page 9]{Kly91}}\label{Thm211} We have
\[
\mathrm{H}^{p}(X, \mathcal{E})_{m}=\mathrm{H}_{n-p}\left(C_{*}(\mathcal{E}, m)\right), \label{eq:general}
\]
where $\mathrm{H}^{p}(X, \mathcal{E})_{m} \text{ is a graded piece of the}$ $M$-graded module $\mathrm{H}^{p}(X, \mathcal{E})$. 
\end{theorem}

\noindent Applying Theorem \ref{eq:general}, we obtain the cohomology groups of projective space as follows:

\begin{equation}
\scalebox{1}{
$\begin{cases}
\begin{aligned}
  &\hspace{0.2cm}\mathrm{H}^{0}(\mathbb{P}^n, \mathcal{E})_m \hspace{0.2cm}= 
  \bigcap_{\rho_{i} \in \Sigma(1)} E^{\mathbf{\rho}_{i}}_{m} \vspace{0.3cm} \\
  &\hspace{0.2cm}\mathrm{H}^{n}(\mathbb{P}^n, \mathcal{E})_m \hspace{0.2cm}= 
  \frac{\mathrm{E}}{\sum\limits_{\rho_{i} \in \Sigma(1)}  E^{\mathbf{\rho}_{i}}_{m}} \label{eq:cp} \vspace{0.3cm} \\
  
  &\hspace{0.2cm}\mathrm{H}^{n-p}(\mathbb{P}^n, \mathcal{E})_m \hspace{0.2cm}= \frac{E^{\mathbf{\rho}_{0}}_{m}\cap...\cap E^{\mathbf{\rho}_{(p-1)}}_{m}\cap \sum\limits_{k \geq p} E^{\mathbf{\rho}_{k}}_{m}}{\sum\limits_{k \geq p} E^{\mathbf{\rho}_{0}}_{m}\cap...\cap E^{\mathbf{\rho}_{(p-1)}}_{m}\cap E^{\mathbf{\rho}_{k}}_{m}} \quad \quad\text{for } 0 < p < n
\end{aligned}
\end{cases}$}
\end{equation}
for each $m\in M$.


\section{\texorpdfstring{$d$-aCM bundles over $\mathbb{P}^2$}{d-aCM bundles over P2}}
\begin{definition}
For a positive integer $d$, a coherent sheaf $\Ee$ on $\mathbb{P}^n$ of dimension $n$ is called {\it $d$-arithmetically Cohen–Macaulay} (abbreviated as ${d}$-aCM) if $\mathrm{H}^{i}(\mathbb{P}^n,\mathcal{E}\otimes\mathcal{O}_{\mathbb{P}^n}(dt))=0$ for all $0 < i < n$ and $t \in \mathbb{Z}$.
\end{definition}

From now on, we focus on the case $n = 2$. The fan of $\mathbb{P}^2$ has three rays, denoted by $\rho_0$, $\rho_1$, and $\rho_2$, whose primitive generators are $u_{\rho_0} = (1,0)$, $u_{\rho_1} = (0,1)$, and $u_{\rho_2} = (-1,-1)$, respectively. We denote the divisors corresponding to $\rho_0$, $\rho_1$, $\rho_2$ by $D_0$, $D_1$, $D_2$ respectively. For each cyclic permutation $(i,j,k)$ of $(0,1,2)$, we let $p_i := D_j \cap D_k$. Since each $D_i$ is a $T$-invariant line, the points $p_i$ are fixed under the torus action on $\mathbb{P}^2$. Then, a vector bundle $\Ee$ on $\PP^2$ is $d$-aCM if and only if we have
\[
\mathrm{H}^1(\PP^2, \Ee\otimes \Oo_{\PP^2}(dt))=0
\]
for all $t\in \ZZ$. From (\ref{eq:cp}) we can get the formula for the first cohomology of the toric bundle $\Ee$.
\\
\setcounter{equation}{1}
\begin{equation}
\mathrm{H}^{1}(\mathbb{P}^2, \mathcal{E})_m = 
\frac{E^{\mathbf{\rho}_{0}}_{m} \cap (E^{\mathbf{\rho}_{1}}_{m} +E^{\mathbf{\rho}_{2}}_{m})}{(E^{\mathbf{\rho}_{0}}_{m} \cap E^{\mathbf{\rho}_{1}}_{m})+(E^{\mathbf{\rho}_{0}}_{m} \cap E^{\mathbf{\rho}_{2}}_{m})} \label{eq:cohop2}
\end{equation}

\begin{remark}\label{rmk:shifting-indices}
Let $\mathcal{E}$ be a non-split toric bundle of rank $2$ on $\mathbb{P}^2$. The corresponding filtrations take the form
\[
E_{\rho_i}(j) =
\begin{cases}
\hspace{0.1cm}\mathrm{E} & \text{for } j \le a_i^2, \\
\hspace{0.1cm}W_i := W_i^1 & \text{for } a_i^2 < j \le a_i^1, \\
\hspace{0.1cm}0 & \text{for } a_i^1 < j,
\end{cases}
\]
for each $i \in \{0,1,2\}$, where the $W_i$ are distinct one-dimensional subspaces of $\mathrm{E}$. Given two collections $\{W_0, W_1, W_2\}$ and $\{V_0, V_1, V_2\}$ of distinct one-dimensional subspaces in $\mathrm{E}$, there exists an automorphism $T \in \mathrm{GL}(\mathrm{E})$ such that $T(W_i) = V_i$ for each $i$. This shows that the isomorphism class of a non-split rank $2$ toric bundle is determined, up to $T$-equivariant isomorphism, by the integers at which the dimension of each filtration drops---from $2$ to $1$, and from $1$ to $0$. For such a toric bundle $\mathcal{E}$, we denote by
\[
\delta(\mathcal{E}) := \left(a_0^2, a_0^1~;~ a_1^2, a_1^1~;~ a_2^2, a_2^1 \right) \in \ZZ^6
\]
the set of indices where the dimension of the filtration changes, and refer to this $6$-tuple as the \emph{shifting indices} of $\mathcal{E}$.
\end{remark}

\begin{remark} 
{\cite[Example 1.2.5]{Kly91}} Let \( a_{i,1} - a_{i,2} = \alpha_i \) for \( i \in (0,1,2) \).  
Then a non-split toric vector bundle \( \mathcal{E} \) on \( \mathbb{P}^2 \) is slope-stable if and only if the three numbers \( \alpha_1, \alpha_2, \alpha_3 \) satisfy the triangle inequality.   
\end{remark}

\begin{lemma}\label{Lemma3.3}
Let $\mathcal{E}$ be a non-split toric bundle of rank $2$ on $\mathbb{P}^2$. Then $\mathrm{H}^{1}(\mathbb{P}^2, \mathcal{E}) \neq 0$ if and only if there exists $(j_0, j_1, j_2)\in \ZZ^{3}$ such that $j_0 + j_1 + j_2 = 0$ and $\mathrm{dim}(E^{\mathbf{\rho}_{i}}(j_i)) = 1$ for all $i\in \{0,1,2\}$.
\end{lemma}

\begin{proof}
Since \(\mathcal{E}\) does not split, the three \(1\)-dimensional vector spaces \(E^{\rho_i}_m\) are mutually distinct for each $m\in M$. Observe that in \eqref{eq:cohop2}, if \(\dim(E^{\rho_i}_m) \in \{0,2\}\) for some \(i \in \{0,1,2\}\),  
then the numerator and denominator are equal, and so \(\mathrm{H}^1(\mathbb{P}^2, \mathcal{E})_m = 0\).  
Hence, for \(\mathrm{H}^1(\mathbb{P}^2, \mathcal{E})\) to be nonzero, there must exist some \( m = (a,b) \in M \) such that \(\dim(E^{\rho_i}_m) = 1\) for all \(i\). Given such \(m = (a,b)\), we compute:
\[
j_0 = \langle m, u_{\rho_0} \rangle = a, \quad
j_1 = \langle m, u_{\rho_1} \rangle = b, \quad
j_2 = \langle m, u_{\rho_2}\rangle = -a - b.
\]
Then \(j_0 + j_1 + j_2 = 0\), and \(E^{\rho_i}_m = E^{\rho_i}(j_i)\) with \(\dim(E^{\rho_i}(j_i)) = 1\) for all \(i\). 

Conversely, suppose that there exists \((j_0, j_1, j_2) \in \mathbb{Z}^3\) such that \(j_0 + j_1 + j_2 = 0\) and \(\dim(E^{\rho_i}(j_i)) = 1\) for all \(i\).  
Define \(m = (j_0, j_1)\). Then we have \(\langle m, \rho_i \rangle = j_i\) for each \(i\), and so \(E^{\rho_i}_m = E^{\rho_i}(j_i)\) with dimension one. Since the subspaces are distinct, the numerator in \eqref{eq:cohop2} exceeds the denominator. In particularm we have \(\mathrm{H}^1(\mathbb{P}^2, \mathcal{E})_m \ne 0\), and so \(\mathrm{H}^1(\mathbb{P}^2, \mathcal{E}) \ne 0\).
\end{proof} The following corollary is an automatic consequence of Example \ref{Twist} and Lemma \ref{Lemma3.3}.

\begin{corollary}\label{cor_3.6}
A non-split toric bundle $\mathcal{E}$ of rank $2$ on $\PP^2$ is $d$-aCM if and only if there does not exist a triple $(j_0, j_1, j_2)\in \ZZ^{3}$ satisfying that
\begin{equation}
j_0 + j_1 + j_2 \equiv 0\mkern-10mu \pmod{d} \text{ with }
\dim(E^{\mathbf{\rho}_{i}}(j_i)) = 1 \quad \mkern-14mu \text{ for } i\in \{0,1,2\}
\end{equation}
\end{corollary}

Note that \( \mathcal{E} \) is $d$-aCM if and only if all of its twists by \( \mathcal{O}_{\mathbb{P}^2}(dt) \) for \( t \in \mathbb{Z} \) is also $d$-aCM. We aim to count the non-split $d$-aCM vector bundles of rank $2$ on $\mathbb{P}^2$ with an $T$-equivariant structure, up to isomorphism and twist by $\mathcal{O}_{\mathbb{P}^2}(dt)$.

\begin{theorem}\label{main}
Let $\mathrm{S}(\PP^2,d;2)$ be the number of non-split $d$-aCM vector bundles of rank $2$  
on $\PP^2$ with $T$-equivariant structure, up to twist by $\mathcal{O}_{\PP^2}(dt)$ for $t \in \mathbb{Z}$.  
Then,
\[
\mathrm{S}(\PP^2,d;2)=\frac{(d-1)d(d+1)(d+2)}{24}.
\]
\end{theorem}

\begin{proof}
We begin by imposing some normalizations on the shifting indices of $\mathcal{E}$. As observed in Example~\ref{Twist}, tensoring with a $T$-equivariant line bundle shifts each filtration by the corresponding coefficients in the divisor.  
In particular, tensoring with $\mathcal{O}_{\mathbb{P}^2}(\ell_0 D_0 + \ell_1 D_1 + \ell_2 D_2)$ increases the shifting indices along $\rho_i$ by $\ell_i$.  
Since tensoring with $\mathcal{O}_{\mathbb{P}^2}(dt)$ does not affect whether a vector bundle is $d$-aCM, we may adjust the shifting indices so that
\[
a_i^2 \ge -1 \, \text{for all } i \in \{0,1,2\}.
\]

\noindent Automatically, this implies that $a_i^1\ge 0$ for all $i$. Futhermore, we may minimize the sum \(a_0^2+a_1^2+a_2^2\) as much as possible, subject to the constraint $a_i^2 \ge -1 \,\text{for all } i \in \{0,1,2\}.$ If this sum is greater than or equal to $d-3$, then there exists a triple  \( (b_0, b_1, b_2) \in \mathbb{Z}_{\ge 0}^3 \) such that \( b_0 + b_1 + b_2 = d \)  and \( a_i^2 - b_i \ge -1 \) for all \( i \). Therefore, by tensoring with the line bundle \( \mathcal{O}_{\mathbb{P}^2}(-b_0D_0-b_1D_1-b_2D_2) \), we may replace \( a_i^2 \) with \( a_i^2 - b_i \), resulting in new indices with strictly smaller total sum \( a_0^2 + a_1^2 + a_2^2 \) that still satisfies the constraint \( a_i^2 \ge -1 \). This shows that we can always achieve 
\[ 
a_0^2 + a_1^2 + a_2^2 \le d - 4. 
\]
Next, we turn to the upper indices. Recall that the $1$-dimensional subspaces in each filtration occur at indices \( j_i \) satisfying \( a_i^2 < j_i \le a_i^1 \), by the definition of the shifting indices.  Suppose, for contradiction, that $a_0^1 + a_1^1 + a_2^1 \ge d$. Then, since we have already arranged so that \( a_0^2 + a_1^2 + a_2^2 \le d - 4 \), it follows that
\[
a_0^2 + a_1^2 + a_2^2 +3 < d \le a_0^1 + a_1^1 + a_2^1.
\]
Therefore, there exists a triple \( (j_0, j_1, j_2) \in \mathbb{Z}^3 \) such that  
\[
a_i^2 < j_i \le a_i^1 \quad \text{for all } i, \quad \text{and } j_0 + j_1 + j_2 = d.
\]

\noindent By Corollary~\ref{cor_3.6}, the existence of such a triple implies that $\mathcal{E}$ is not \( d \)-aCM. Hence, we conclude that $a_0^1 + a_1^1 + a_2^1 \le d-1$. Moreover, if all \( a_i^2 = -1 \), there exists a triple \( (j_0, j_1, j_2) = (0,0,0) \) such that \( \dim(E^{\rho_i}(j_i)) = 1 \) for all $i\in \{0,1,2\}$, which also violates the \( d \)-aCM condition by Corollary \ref{cor_3.6}. We define $\overline{\mathbf{SI}}(d)$ to be the set of all shifting indices $
(a_0^2, a_0^1\,;\, a_1^2, a_1^1\,;\, a_2^2, a_2^1) \in \mathbb{Z}^6
$
satisfying the following:
\begin{itemize}
    \item[$\bullet$]  \( a_i^2 \ge -1 \) (and so \( a_i^1 \ge 0 \)) for all \( i \in \{0, 1, 2\} \);
    \item[$\bullet$] \( a_0^1 + a_1^1 + a_2^1 \le d - 1 \) (and so \(a_0^2 + a_1^2 + a_2^2 \le d - 4\)) ;
    \item[$\bullet$] there exist \( i \in \{0,1,2\} \) such that \( a_i^2 > -1 \).
\end{itemize}
It is clear that for each $\delta$ $\in \overline{\mathbf{SI}}(d)$ the corresponding toric vector bundle is $d$-aCM. Finally, suppose that we tensor \( \mathcal{E} \) with the principal line bundle \( \mathcal{O}_{\mathbb{P}^2}(c_0 D_0 + c_1 D_1 + c_2 D_2)\) where \( c_0 + c_1 + c_2 = 0 \) and \( (c_0, c_1, c_2) \ne (0,0,0) \). Then the resulting bundle is isomorphic to \( \mathcal{E} \) as a vector bundle, but its shifting indices are altered by \( c_i\) along each ray \( \rho_i\). To avoid redundancy from such shifts, we now define an equivalence relation on shifting indices. For two shifting indices $\delta = (a_0^2, a_0^1\,;\, a_1^2, a_1^1\,;\, a_2^2, a_2^1)$ and $\delta' = ({a_0^2}', {a_0^1}'\,;\, {a_1^2}', {a_1^1}'\,;\, {a_2^2}', {a_2^1}')$ in $\overline{\mathbf{SI}}(d)$, we set $\delta \sim \delta'$ if there exists $(c_0,c_1,c_2) \in \mathbb{Z}^3 $ such that $c_0+c_1+c_2 = 0$ and 
\[ 
a_j^i + c_j = {a_j^i}' 
\]
for all \( j \in \{0,1,2\} \) and \( i \in \{1,2\} \).
The relation $ \sim $ is an equivalence relation on $\overline{\mathbf{SI}}(d)$. Define 
\[
\mathbf{SI}(d) := 
\left\{\, \delta \in \overline{\mathbf{SI}}(d) \;\middle|\;
\begin{array}{l}
\text{$\delta$ maximizes $a_2^2$ in its equivalence class $[\delta]$}
\end{array}
\,\right\}.
\]
Each element of $\mathbf{SI}(d)$ then corresponds bijectively to a non-split $d$-aCM toric vector bundle of rank $2$ on $\mathbb{P}^2$, up to twist by $\mathcal{O}_{\mathbb{P}^2}(dt)$. Note that by maximizing $a_2^2$ within each equivalence class, the remaining lower indices $a_0^2$ and $a_1^2$ are necessarily minimized, and hence must equal $-1$.
In particular, we have 
\[
\mathrm{S}(\PP^2,d;2)=|\mathbf{SI}(d)|.
\]
We present the following representative examples:
\begin{align*}
\overline{\mathbf{SI}}(2) &= \left\{\, 
(-1, 0;\, -1, 0;\, 0, 1), \quad
(-1, 0;\, 0, 1;\, -1, 0), \quad
(0, 1;\, -1, 0;\, -1, 0)
\,\right\}, \\
\mathbf{SI}(2) &= \left\{\, (-1, 0;\, -1, 0;\, 0, 1) \,\right\}, \\
\mathbf{SI}(3) &= \left\{\, 
\begin{aligned}
&(-1, 0;\, -1, 0;\, 0, 1), \quad
(-1, 0;\, -1, 0;\, 0, 2), \quad
(-1, 0;\, -1, 0;\, 1, 2),\\
&(-1, 0;\, -1, 1;\, 0, 1), \quad
(-1, 1;\, -1, 0;\, 0, 1)
\end{aligned}
\,\right\},
\end{align*}
and so we get $\mathrm{S}(\PP^2,2;2)=1$ and $\mathrm{S}(\PP^2,3;2)=5$. Some of these shifting indices are illustrated in the tables below.

\renewcommand{\arraystretch}{1.34}
\renewcommand{\arraystretch}{1.5}
\begin{table}[h]
\centering
\begin{tabular}{
    >{\centering\arraybackslash}m{1.4cm} |
    >{\centering\arraybackslash}m{1.6cm} |
    >{\centering\arraybackslash}m{1.6cm} |
    >{\centering\arraybackslash}m{1.6cm} |
    >{\centering\arraybackslash}m{1.6cm} |
    >{\centering\arraybackslash}m{1.6cm} |
    >{\centering\arraybackslash}m{1.6cm} |
    >{\centering\arraybackslash}m{1.6cm} |
}
\hline
 & $\ldots$ & $-1$ & $0$ & $1$ & $2$ & $3$ & $\ldots$ \\
\hline
$\rho_0$ &
$\ldots$ &
\cellcolor{lightred}$\mathrm{E}$ &
\cellcolor{lightblue}$W_0^1$ &
0 & 0 & 0 &
$\ldots$ \\
\hline
$\rho_1$ &
$\ldots$ &
\cellcolor{lightred}$\mathrm{E}$ &
\cellcolor{lightblue}$W_1^1$ &
0 & 0 & 0 &
$\ldots$ \\
\hline
$\rho_2$ &
$\ldots$ &
\cellcolor{lightred}$\mathrm{E}$ &
\cellcolor{lightred}$\mathrm{E}$ &
\cellcolor{lightblue}$W_2^1$ &
0 & 0 &
$\ldots$ \\
\hline
\end{tabular}
\caption{$\delta(\Ee) = (-1,0~;~-1,0~;~0,1)$}
\label{tab:table1}
\end{table}

\begin{table}[h]
\centering
\begin{tabular}{
    >{\centering\arraybackslash}m{1.4cm} |
    >{\centering\arraybackslash}m{1.6cm} |
    >{\centering\arraybackslash}m{1.6cm} |
    >{\centering\arraybackslash}m{1.6cm} |
    >{\centering\arraybackslash}m{1.6cm} |
    >{\centering\arraybackslash}m{1.6cm} |
    >{\centering\arraybackslash}m{1.6cm} |
    >{\centering\arraybackslash}m{1.6cm} |
}
\hline
 & $\ldots$ & $-1$ & $0$ & $1$ & $2$ & $3$ & $\ldots$ \\
\hline
$\rho_0$ &
$\ldots$ &
\cellcolor{lightred}$\mathrm{E}$ &
\cellcolor{lightblue}$W_0^1$ &
0 & 0 & 0 &
$\ldots$ \\
\hline
$\rho_1$ &
$\ldots$ &
\cellcolor{lightred}$\mathrm{E}$ &
\cellcolor{lightblue}$W_1^1$ &
0 & 0 & 0 &
$\ldots$ \\
\hline
$\rho_2$ &
$\ldots$ &
\cellcolor{lightred}$\mathrm{E}$ &
\cellcolor{lightred}$\mathrm{E}$ &
\cellcolor{lightred}$\mathrm{E}$ &
\cellcolor{lightblue}$W_2^1$ &
0 &
$\ldots$ \\
\hline
\end{tabular}
\caption{$\delta(\Ee) = (-1,0~;~-1,0~;~1,2)$}
\label{tab:table2}
\end{table}

\begin{table}[h]
\centering
\begin{tabular}{
    >{\centering\arraybackslash}m{1.4cm} |
    >{\centering\arraybackslash}m{1.6cm} |
    >{\centering\arraybackslash}m{1.6cm} |
    >{\centering\arraybackslash}m{1.6cm} |
    >{\centering\arraybackslash}m{1.6cm} |
    >{\centering\arraybackslash}m{1.6cm} |
    >{\centering\arraybackslash}m{1.6cm} |
    >{\centering\arraybackslash}m{1.6cm} |
}
\hline
 & $\ldots$ & $-1$ & $0$ & $1$ & $2$ & $3$ & $\ldots$ \\
\hline
$\rho_0$ &
$\ldots$ &
\cellcolor{lightred}$\mathrm{E}$ &
\cellcolor{lightblue}$W_0^1$ &
0 & 0 & 0 &
$\ldots$ \\
\hline
$\rho_1$ &
$\ldots$ &
\cellcolor{lightred}$\mathrm{E}$ &
\cellcolor{lightblue}$W_1^1$ &
0 & 0 & 0 &
$\ldots$ \\
\hline
$\rho_2$ &
$\ldots$ &
\cellcolor{lightred}$\mathrm{E}$ &
\cellcolor{lightred}$\mathrm{E}$ &
\cellcolor{lightblue}$W_2^1$ &
\cellcolor{lightblue}$W_2^1$ &
0 &
$\ldots$ \\
\hline
\end{tabular}
\caption{$\delta(\Ee) = (-1,0~;~-1,0~;~0,2)$}
\label{tab:table3}
\end{table}

Each row in the table is indexed by the ray $\rho_i$, and the numbers at the top of each column indicate the filtration indices. Table~1 shows the filtration corresponding to the case $d = 2$. As \( d \) increases from \( 2 \) to \( 3 \), the filtration patterns fall into three distinct types:
\begin{itemize}
    \item [(i)] the case when $d = 2$, 
    \item[(ii)] adding a cell with $\mathrm{E}$ to Table~1 (see Table~2), or
    \item[(iii)] adding a cell with $1$-dimensional vector space to Table~1 (see Table~3).
\end{itemize}
As $d$ increases by one, the corresponding shifting indices may change in one of two ways: either both $a_2^1$ and $a_2^2$ increase by $1$, or a single $a_i^1$ increases by $1$ for some $i \in \{0,1,2\}$. The former case corresponds to adding a cell containing $\mathrm{E}$ in the table, which is equivalent to tensoring the bundle by the line bundle $\mathcal{O}_{\mathbb{P}^2}(D_2)$.
The latter case corresponds to inserting a $1$-dimensional subspace, yielding three distinct toric bundles depending on the ray. These bundles are not isomorphic to each other. Indeed, two toric bundles that are isomorphic as vector bundles can differ only by a twist by a principal $T$-equivariant line bundle (see Remark~\ref{rmk_210}), which does not affect the number of $1$-dimensional subspaces in each filtration. Thus, the five shifting indices in $\mathbf{SI}(3)$ fall into three categories:
the original case when $d = 2$, its twist by $\mathcal{O}_{\mathbb{P}^2}(1)$, and three additional cases corresponding to ordered triples $(a, b, c) \in \mathbb{Z}_{> 0}^3$ satisfying $a + b + c = 4$,  
where each entry denotes the number of $1$-dimensional cells added to the respective filtration.

This trichotomy observed at \( d = 3 \) holds for general \( d \). To see this, we classify each element in  \( \delta \in \mathrm{SI}(d) \) into three types. First, if $a_0^1+a_1^1+a_2^1 < d-1$, then \( \delta \) satisfies all the defining inequalities of \( {\mathbf{SI}}(d-1) \) as well. Second, if \( a_0^1 + a_1^1 + a_2^1 = d - 1 \) and \( a_2^2 > 0 \), let $\delta' = (a_0^2, a_0^1\,;\, a_1^2, a_1^1\,;\, a_2^2 - 1, a_2^1 - 1)$. Then $\delta'$ belongs to $\mathbf{SI}(d-1) \setminus \mathbf{SI}(d-2)$, since the sum $a_0^1 + a_1^1 + (a_2^1 -1)= d - 2$. Moreover, since $\delta$ and $\delta'$ differ by $1$ along both $a_2^2$ and $a_2^1$, the corresponding toric bundles differ by a twist by $\mathcal{O}_{\mathbb{P}^2}(1)$. Third, if \( a_0^1 + a_1^1 + a_2^1 = d - 1\) and $(a_0^2, a_1^2, a_2^2) = (-1, -1, 0)$, then we seek nonnegative integers $a_0^1, a_1^1 \ge 0$ and $a_2^1 > 0$ such that $a_0^1 + a_1^1 + a_2^1 = d - 1$. The number of such solutions is ${}_3\mathrm{H}_{d-2} = \binom{d}{2}$.

Consequently, the non-split \( d \)-aCM toric bundles of rank \( 2 \) on \( \mathbb{P}^2 \) fall into three types: 
\begin{itemize}
    \item[(I)] those already appearing in \( \mathbf{SI}(d-1) \); 
    \item[(II)] those obtained from $\mathbf{SI}(d-1) \setminus \mathbf{SI}(d-2)$ by twisting with \( \mathcal{O}_{\mathbb{P}^2}(1) \),  
    corresponding to the addition of a cell containing \( \mathrm{E} \);
    \item[(III)] those formed by adding a single cell containing a $1$-dimensional subspace to one of the filtrations,  
    counted by \( {}_3\mathrm{H}_{d-2}  = \binom{d}{2} \).
\end{itemize}

\noindent This observation gives rise to the following recurrence relation.
\[
\begin{cases}
\hspace{.2cm}\mathrm{S}(\PP^2,2;2)=1,\quad \mathrm{S}(\PP^2,3;2)=5 \vspace{.3cm}\\ 
\hspace{.2cm}\mathrm{S}(\PP^2,d;2) = \mathrm{S}(\PP^2,d-1;2) + \left\{\mathrm{S}(\PP^2,d-1;2)-\mathrm{S}(\PP^2,d-2;2)\right\} + {}_{3}\mathrm{H}_{d-2} \quad \text{ for } d\geq4 \vspace{.1cm}
\end{cases}
\]
\noindent Solving this recurrence relation yields the following
\[
\mathrm{S}(\PP^2,d;2) = \sum_{i=1}^{d-1} \sum_{j=1}^{i} \frac{j(j+1)}{2} =  \frac{(d-1)d(d+1)(d+2)}{24}.\qedhere
\] 
\end{proof}

\begin{remark}
In the proof of Theorem \ref{main}, one can observe that a $d$-aCM toric bundle over $\PP^2$ is also $d'$-aCM for all $d'>d$. This phenomenon is generally uncommon.
\end{remark}

According to \cite[Theorem 8.2]{Per03}, for the toric bundles $\mathcal{E}$ of rank $2$ over the $\mathbb{P}^2 $, the following short exact sequence holds:

\begin{equation}
0 \to \mathcal{O}_{\PP^2}(a_{0}^{2}D_{0}+a_{1}^{2}D_{1}+a_{2}^{2}D_{2}) \stackrel{}{\to} \bigoplus_{(i,j,k) \in A}^{} \mathcal{O}_{\PP^2}(a_{i}^{2}D_{i}+a_{j}^{2}D_{j}+a_{k}^{1}D_{k}) \to \mathcal{E} \to 0,
\label{eq:seq}
\end{equation}

\noindent where each $a_{i}^{n}$ are the indices of filtrations and $A=\{(0,1,2),(1,2,0),(2,0,1)\}$.
To further investigate the bundles, one can analyze using \eqref{eq:seq}.

\begin{example}
For $d = 2$, the only possible class is a twist of the tangent sheaf, $\mathcal{T}_{\mathbb{P}^2}(-2)$, as shown in Example \ref{Twist} and Example \ref{Tangentsheaf}. In this case, the sequence obtained from short exact sequence $\eqref{eq:seq}$ is  
\[
0 \longrightarrow \mathcal{O}_{\mathbb{P}^2}(-D_0-D_1) \longrightarrow \mathcal{O}_{{\mathbb{P}^2}}(-D_0-D_1+D_2){\oplus}{\mathcal{O}_{\mathbb{P}^2}(-D_0)}{\oplus}{\mathcal{O}_{\mathbb{P}^2}(-D_1)} \longrightarrow \mathcal{E} \longrightarrow 0,
\]
which coincides with the Euler sequence tensored by \( \mathcal{O}_{\mathbb{P}^2}(-2) \).

\end{example}

\begin{example}\label{Err}
For the case $d=3$, we have $\mathrm{S}(\PP^2, 3; 2)=5$. Two of the five cases are 
$\mathcal{T}_{\mathbb{P}^2}(-2)$ and $\mathcal{T}_{\mathbb{P}^2}(-1)$, 
corresponding to Table~\ref{tab:table1} and Table~\ref{tab:table3}, respectively.
The other three cases are symmetric and so we assume that $\delta(\Ee) = \left(-1, 0~;~ -1, 0~;~ 0, 2 \right)$ to have the following exact sequence:

\[
0 \to \mathcal{O}_{\PP^2}(-D_{0}-D_{1}) \stackrel{}{\to} \mathcal{O}_{\PP^2}(-D_{0}-D_{1}+2D_{2})\oplus \mathcal{O}_{\PP^2}(-D_{0}) \oplus \mathcal{O}_{\PP^2}(-D_{1}) \to\mathcal{E} \to 0.
\]

\noindent By adding a short exact sequence to the given sequence, we get the following diagram:

\begin{center}
\begin{tikzcd}[
 column sep=small, row sep=normal,
  ar symbol/.style = {draw=none,"\textstyle#1" description,sloped},
  isomorphic/.style = {ar symbol={\cong}}
  ]
& & 0 \ar[d]& 0 \ar{d} &  \\
& & \mathcal{O}_{\PP^2}(-D_{0}-D_{1}+2D_{2}) \ar[r, isomorphic] \ar[d] & \mathcal{O}_{\PP^2}(-D_{0}-D_{1}+2D_{2})  \ar[d] &  \\
0 \ar[r] & \mathcal{O}_{\PP^2}(-D_{0}-D_{1}) \ar[r] \ar[d,isomorphic] & \left(
\begin{array}{c}
\mathcal{O}_{\PP^2}(-D_{0}-D_{1}+2D_{2}) \\
\oplus \\
\mathcal{O}_{\PP^2}(-D_{0}) \\
\oplus \\
\mathcal{O}_{\PP^2}(-D_{1})
\end{array}\right) \ar[r] \ar[d] & \mathcal{E} \ar[r] \ar[d] & 0 \\
0 \ar[r] & \mathcal{O}_{\PP^2}(-D_{0}-D_{1}) \ar[r] & \left(
\begin{array}{c}
\mathcal{O}_{\PP^2}(-D_{0}) \\
\oplus \\
\mathcal{O}_{\PP^2}(-D_{1})
\end{array}\right) \ar[r] \ar[d] & \mathcal{I}_{p_2, \PP^2} \ar[r] \ar[d] & 0 \\
& & 0 & 0. 
\end{tikzcd}
\end{center}

\noindent Recall that the point \( p_2 \) in the diagram above represents the intersection of \( D_0 \) and \( D_1 \). Since 
\[
\dim \left( \mathrm{Ext}_{}^1 (\mathcal{I}_{p_2, \mathbb{P}^2}, \mathcal{O}_{\mathbb{P}^2}) \right) = 1,
\]
the bundle \( \mathcal{E} \) is uniquely determined by the point \( p_2 \). There are two other possible configurations involving four one-dimensional cells. In these alternative cases, we still obtain a similar sequence, but the point \( p_2 \) is replaced by $p_0$ and $p_1$, respectively. Recall that the moduli space $\mathbf{M}_{\PP^2}(0,1)$ of slope-semistable vector bundles of rank two with the Chern classes $(c_1, c_2)=(0,1)$ isomorphic to $\PP^2$ and each slope-semistable sheaf is a nontrivial extension of $\Ii_{p, \PP^2}$ by $\Oo_{\PP^2}$; see \cite[Theorem 8.1]{Angelini}. 
\end{example}

\begin{example}
In the case of $d=4$ there are $15$ classes that satisfy the conditions. However, nine of them either belong to cases with
$d<4$ or are their twists. More precisely, the newly added bundles are those with a $1$-dimensional vector space contained in five cells, and there are six such cases. Furthermore, these six cases can be categorized into two groups based on how the 1-dimensional vector spaces are distributed among the three filtrations. In one group, two of the filtrations contain two 1-dimensional cells each, while the remaining filtration contains one. In the other group, one filtration contains three 1-dimensional cells, while the other two contain one each. Consider the bundle $\mathcal{E}$, corresponding to the tuple  \{(-1,0),(-1,1),(0,2)\} as an example from the first group. By constructing the sequence in the same way as in Example \ref{Err}, we obtain the following.

\begin{tikzcd}[
 column sep=small, row sep=normal,
  ar symbol/.style = {draw=none,"\textstyle#1" description,sloped},
  isomorphic/.style = {ar symbol={\cong}}
  ]
& & 0 \ar[d]& 0 \ar{d} &  \\
& & \mathcal{O}_{\PP^2}(-D_{1}) \ar[r, isomorphic] \ar[d] & \mathcal{O}_{\PP^2}(-D_{1})  \ar[d] &  \\
0 \ar[r] & \mathcal{O}_{\PP^2}(-D_{0}-D_{1}) \ar[r] \ar[d,isomorphic] & \left(
\begin{array}{c}
\mathcal{O}_{\PP^2}(-D_{0}-D_{1}+2D_{2}) \\
\oplus \\
\mathcal{O}_{\PP^2}(-D_{0}+D_{1}) \\
\oplus \\
\mathcal{O}_{\PP^2}(-D_{1})
\end{array}\right) \ar[r] \ar[d] & \mathcal{E} \ar[r] \ar[d] & 0 \\
0 \ar[r] & \mathcal{O}_{\PP^2}(-D_{0}-D_{1}) \ar[r] & \left(
\begin{array}{c}
\mathcal{O}_{\PP^2}(-D_{0}-D_{1}+2D_{2}) \\
\oplus \\
\mathcal{O}_{\PP^2}(-D_{0}+D_{1})
\end{array}\right) \ar[r] \ar[d] & {\mathcal{I}_{Z_1, \PP^2} \tensor {\mathcal{O}_{\PP^2}(-D_{0}+D_{1}+2D_{2})}}  \ar[r] \ar[d] & 0 \\
& & 0 & 0 .
\end{tikzcd}

\noindent Here, the zero-dimensional subscheme $Z_1$ is given as the intersection $2D_1 \cap 2D_2$, i.e. $Z_1=4p_0$. Noticing that $\Ee$ is a stable vector bundle of rank $2$ with the Chern classes $(c_1, c_2)=(1,2)$, we get that the corresponding set of jumping lines of the second kind
\[
C(\Ee)=\left\{\ell \in (\PP^2)^*~\big|~ \mathrm{h}^0\left(\Ee(-1)_{|\ell^2}\right)>0\right\}
\]
is a conic with corank one; see \cite{Hulek}. By tensoring the last vertical sequence in the diagram above with $\Oo_{\ell^2}(-1)$ with $\ell=D_1$ or $D_2$, one gets
\[
0\to \Oo_{\ell^2} \to \Ee(-1)_{|\ell^2}\to \Oo_{\ell^2}(-1) \to 0.
\]
It implies that the dual points $[D_0], [D_1] \in C(\Ee)\subset (\PP^2)^*$. Since $\Ee$ is a toric bundle, the locus $C(\Ee)$ is a $T$-invariant divisor of $(\PP^2)^*$ and so we get $C(\Ee)=[p_1] \cup [p_2]$, where $[p_i]\subset (\PP^2)^*$ is the dual line corresponding to $p_i$. 
\end{example}

\end{document}